\documentclass[11pt]{article}
\usepackage{geometry}
\usepackage{amssymb, amsmath, amsthm, fullpage, color, tikz,microtype}
\newtheorem{theorem}{Theorem}[section]
\newtheorem{lemma}[theorem]{Lemma}
\newtheorem{definition}[theorem]{Definition}

\numberwithin{equation}{section}

\newcommand{\disjointunion}{{\stackrel{\centerdot}{\cup}}}

\begin{document}

\title{Parking Cars after a Trailer}
\author{Richard Ehrenborg and Alex Happ}

\date{}

\maketitle

\begin{abstract}
Recently, the authors 
extended the notion of parking functions
to parking sequences, which include
cars of different sizes, and proved
a product formula for the number of such sequences.
We here give a refinement of that
result involving parking the cars after a trailer.
The proof of the refinement uses a multi-parameter
extension of the Abel--Rothe polynomial
due to Strehl.
\end{abstract}

\section{The result}

Parking sequences were
introduced in~\cite{Ehrenborg_Happ}
as an extension of the classical notion
of parking functions, where we now take into account
parking cars of different sizes.
This extension differs from other
extensions of parking
functions
~\cite{Chebikin_Postnikov,Kung_Yan_I,Kung_Yan_II,Kung_Yan_III,Yan}
since
the parking sequences are not
invariant under permuting the entries.
The main result in~\cite{Ehrenborg_Happ}
is that the number of
parking sequences is given by the product
\begin{equation}
(y_{1}+n)
\cdot (y_{1}+y_{2}+n-1)
\cdots (y_{1}+\cdots+y_{n-1}+2),
\label{equation_no_trailer}
\end{equation}
where the $i$th car has length $y_{i}$.
Note that this reduces to the classical
$(n+1)^{n-1}$ result of 
Konheim and Weiss~\cite{Konheim_Weiss}
when setting $y_{1} = y_{2} = \cdots = y_{n} = 1$.
The proof in~\cite{Ehrenborg_Happ}
is an extension of the circular argument by
Pollak; see~\cite{Riordan}.

We now introduce a refinement of the result
by adding a trailer.
\begin{definition}
Let there be $n$ cars $C_{1},\ldots,C_n$
of sizes $y_{1},\ldots,y_{n}$, 
where $y_{1}, \ldots, y_{n}$ are positive integers.
Assume there are $z-1 + \sum_{i=1}^{n} y_{i}$ spaces in a row,
where the trailer occupies the $z-1$ first spaces.
Furthermore, let car $C_{i}$ have the preferred spot~$c_{i}$.
Now let the cars
in the order~$C_{1}$ through $C_{n}$
park according to the following rule:
\begin{quote}
{\normalsize 
Starting at position $c_{i}$,
car $C_{i}$
looks for the first empty spot $j \geq c_{i}$.
If the spaces~$j$ through $j+y_{i}-1$ are empty, then car $C_{i}$
parks in these spots.
If any of the spots
$j+1$ through $j+y_{i}-1$ is already occupied,
then there will be a collision, and the result is not a parking sequence.
}
\end{quote}
Iterate this rule for all the cars
$C_{1}, C_{2}, \ldots, C_{n}$.
We call $(c_{1},\ldots, c_n)$ a \emph{parking sequence}
for $\vec{y}=(y_{1},\ldots,y_{n})$
if all $n$ cars can park without any collisions
and without leaving
the $z-1 + \sum_{i=1}^{n} y_{i}$ parking spaces.
\end{definition}

As an example, consider three cars of sizes
$\vec{y}=(2,2,1)$,
a trailer of size $3$, that is $z=4$,
and the preferences $\vec{c}=(5,6,2)$.
Then there are $2+2+1=5$ available parking spaces after the trailer,
and the final configuration of the cars is
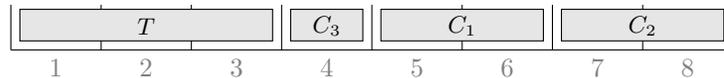
\[\begin{tikzpicture}[scale=1.2]
\draw (0,0) -- (8,0);
\foreach \x in {0,...,8}
\draw (\x,0) -- (\x,0.5);
\foreach \x in {1,...,8}
\node[gray] at (\x-0.5,-0.2) {\small$\x$};
\draw[fill=gray!20] (0.1,0.1) rectangle (2.9,0.45);
\draw[fill=gray!20] (3.1,0.1) rectangle (3.9,0.45);
\draw[fill=gray!20] (4.1,0.1) rectangle (5.9,0.45);
\draw[fill=gray!20] (6.1,0.1) rectangle (7.9,0.45);
\node at (1.5,0.265) {\footnotesize $T$};
\node at (3.5,0.265) {\footnotesize $C_{3}$};
\node at (5,0.265) {\footnotesize $C_{1}$};
\node at (7,0.265) {\footnotesize $C_{2}$};
\end{tikzpicture}\]
All cars are able to park, so this yields a parking sequence.

We now have the main result.
Observe that when setting $z=1$, this expression reduces
to equation~\eqref{equation_no_trailer}.
\begin{theorem}
The number of parking sequences $f(\vec{y};z)$ for car sizes
$\vec{y}=(y_{1},\ldots,y_{n})$
and a trailer of length $z-1$
is
given by the product
\[
f(\vec{y};z)=
z \cdot (z+y_{1}+n-1)
\cdot (z+y_{1}+y_{2}+n-2)
\cdots (z+y_{1}+\cdots+y_{n-1}+1).\]
\label{theorem_parking}
\end{theorem}

\section{The proof}

The first part of our proof comes
from the following identity.
Let $\disjointunion$
denote disjoint union of sets.

\begin{lemma}
    The number of parking sequences for car sizes 
    $(y_{1},\ldots,y_{n},y_{n+1})$ 
    and a trailer of length $z-1$ 
    satisfies the recurrence
        \[f(\vec{y},y_{n+1};z)
        =
        \sum_{L \disjointunion R=\{1,\ldots,n\}}
        \left(z+\sum_{l\in L} y_l\right)
        \cdot f(\vec{y}_L;z) \cdot f(\vec{y}_R;1),\]
    where $\vec{y}_S=(y_{s_1},\ldots,y_{s_k})$ for 
    $S=\{s_{1}<s_{2}<\cdots<s_{k}\}\subseteq\{1,\ldots,n\}$.
 \label{lemma}
\end{lemma}
\begin{proof} 
Consider the situation required 
for the last car $C_{n+1}$ to park successfully:
   \begin{itemize}
   	\item[--]
	 Car $C_{n+1}$ must see, 
	   to the left of its vacant spot, 
	   the trailer along with a subset of the cars 
	   labeled with indices $L$ occupying the first 
	   $z-1+\sum_{l\in L} y_l$ spots.
           Hence, the restriction $\vec{c}_L$ 
	   of $\vec{c}=(c_1,c_2,\ldots,c_{n+1})$ 
	   to the indices in $L$ 
	   must be a parking sequence for $\vec{y}_L$ 
	   and trailer of length $z-1$.
	   This can be done in $f(\vec{y}_L;z)$ possible ways.
	\item[--]
	 Car  $C_{n+1}$ must have a preference $c_{n+1}$ 
	   that lies in the range $[1,z+\sum_{l\in L}y_l]$.
	\item[--]
	 Car $C_{n+1}$ must see, 
	   to the right of its vacant spot, 
	   the complementary subset of cars 
	   labeled with indices $R=\{1,2,\ldots,n\}- L$ 
	   occupying the last $\sum_{r\in R} y_r$ spots. 
	   These cars must have parked successfully 
	   with preferences $\vec{c}_R$ and no trailer, that is,
	   $z=1$. This is enumerated by $f(\vec{y}_R;1)$.
   \end{itemize}
Now summing over all decompositions
$L \disjointunion R=\{1,2,\ldots,n\}$, the recursion follows.
\end{proof}

The next piece of the proof of Theorem~\ref{theorem_parking}
utilizes a multi-parameter
convolution identity
due to Strehl~\cite{Strehl}.
Let
${\bf x} = (x_{i,j})_{1 \leq i < j}$
and
${\bf y} = (y_{j})_{1 \leq j}$ be two infinite sets of parameters.
For a finite subset~$A$ of the positive integers,
define the two sums
$$
{\bf x}_{> a}^{A}
=
\sum_{j \in A, j > a} x_{a,j}
\:\:\:\: \text{ and }\:\:\:\:
{\bf y}_{\leq a}^{A}
=
\sum_{j \in A, j \leq a} y_{j} .
$$
Define the polynomials
$t_{A}({\bf x}, {\bf y}; z)$
and
$s_{A}({\bf x}, {\bf y}; z)$
by
\begin{align*}
t_{A}({\bf x}, {\bf y}; z)
& =
z \cdot
\prod_{a \in A-\max(A)}
(z + {\bf y}_{\leq a}^{A} + {\bf x}_{> a}^{A}) ,
\\
s_{A}({\bf x}, {\bf y}; z)
& =
\prod_{a \in A}
(z + {\bf y}_{\leq a}^{A} + {\bf x}_{> a}^{A}) .
\end{align*}
Note that, when $A$ is the empty set,
we set
$t_{A}({\bf x}, {\bf y}; z)$ to be $1$.
We directly have that
\begin{equation}
(z + {\bf y}_{\leq \max(A)}^{A})
\cdot
t_{A}({\bf x}, {\bf y}; z)
 =
z
\cdot
s_{A}({\bf x}, {\bf y}; z) .
\label{equation_easy_one}
\end{equation}

Now Theorem~1, equation~(6)
in~\cite{Strehl} states:
\begin{theorem}[Strehl]
The polynomials
$s_{L}({\bf x}, {\bf y}; z)$
and
$t_{R}({\bf x}, {\bf y}; w)$
satisfy the following convolution identity:
\begin{equation}
s_{A}({\bf x}, {\bf y}; z+w)
=
\sum_{L \disjointunion R = A}
s_{L}({\bf x}, {\bf y}; z)
\cdot
t_{R}({\bf x}, {\bf y}; w)   .
\label{equation_Sheffer}
\end{equation}
\label{theorem_Strehl}
\end{theorem}

Strehl first interprets 
$s_A({\bf x},{\bf y};z)$ 
and 
$t_A({\bf x},{\bf y};z)$ 
as sums of weights on functions, 
then translates these via a bijection 
to sums of weights on rooted, labeled trees 
where the $x_{i,j}$'s record ascents, 
and the $y_j$'s record descents. 
The proof of~(\ref{equation_Sheffer}) 
then follows from the structure 
inherent in splitting a tree into two. 
A similar result using the same bijection 
was discovered by E\v{g}ecio\v{g}lu and Remmel
in~\cite{E_Remmel}.

\begin{proof}[Proof of Theorem~\ref{theorem_parking}]
The proof follows
from noticing that our proposed expression for 
$f(\vec{y};z)$ is Strehl's polynomial
$t_{\{1,2,\ldots,n\}}({\bf1},{\bf y};z)$. 
By induction we obtain
\begin{align*}
f(\vec{y},y_{n+1};z)
& =
        \sum_{L \disjointunion R=\{1,2,\ldots,n\}}
        \left(z+\sum_{l\in L} y_l\right)
        \cdot f(\vec{y}_L;z)\cdot f(\vec{y}_R;1) \\
& =
\sum_{L \disjointunion R = \{1,2,\ldots,n\}}
(z + {\bf y}_{\leq \max(L)}^{L})
\cdot
t_{L}({\bf 1}, {\bf y}; z)
\cdot
t_{R}({\bf 1}, {\bf y}; 1) \\
& =
\sum_{L \disjointunion R = \{1,2,\ldots,n\}}
z
\cdot
s_{L}({\bf 1}, {\bf y}; z)
\cdot
t_{R}({\bf 1}, {\bf y}; 1) \\
& =
z
\cdot
s_{\{1,2,\ldots,n\}}({\bf 1}, {\bf y}; z+1) \\
& =
t_{\{1,2,\ldots,n+1\}}({\bf 1}, {\bf y}; z) ,
\end{align*}
where we used 
the recursion in
Lemma~\ref{lemma},
equation~\eqref{equation_easy_one} 
and
Theorem~\ref{theorem_Strehl}.
\end{proof}

\section{Concluding remarks}

The polynomial $t_{A}({\bf x}, {\bf y}; z)$
satisfies the following convolution identity;
see~\cite[Equation~(7)]{Strehl},
\begin{equation}
t_{A}({\bf x}, {\bf y}; z+w)
=
\sum_{B \disjointunion C = A}
t_{B}({\bf x}, {\bf y}; z)
\cdot
t_{C}({\bf x}, {\bf y}; w)   .
\label{equation_binomial}
\end{equation}
Hence it is suggestive to think of this
polynomial as of binomial type
and the polynomial
$s_{A}({\bf x}, {\bf y}; w)$ as an
associated Sheffer sequence;
see~\cite{Rota_Kahaner_Odlyzko}.
When setting all the parameters ${\bf x}$
to be constant and also the parameters
${\bf y}$ to be constant, we obtain
the classical
Abel--Rothe polynomials.
Hence it is natural to ask if
other sequences of binomial type and
their associated Sheffer sequences
have multi-parameter extensions.
Since the Hopf algebra ${\bf k}[x]$
explains sequences of binomial type,
one wonders if there is a Hopf algebra
lurking in the background
explaining equations~\eqref{equation_binomial}
and~\eqref{equation_Sheffer}.

\vspace{-2mm}
\section*{Acknowledgment}

Both authors were partially supported by
National Security Agency grant~H98230-13-1-0280.
The first author wishes to thank the Mathematics Department of
Princeton University where this work was completed.

\newcommand{\journal}[6]{#1, #2, {\it #3} {\bf #4} (#5), #6.}
\newcommand{\toappear}[3]{#1, #2, to appear in {\it #3}.}

\noindent 
{\em Department of Mathematics,
University of Kentucky,
Lexington, KY 40506,} \\
{richard.ehrenborg@uky.edu}, {alex.happ@uky.edu}


\small

\begin{thebibliography}{99}

\bibitem{Chebikin_Postnikov}
\journal{D.\ Chebikin and A.\ Postnikov}
            {Generalized parking functions, descent numbers,
              and chain polytopes of ribbon posets}
            {Adv.\ in Appl.\ Math.}
            {44}{2010}{145--154}


\bibitem{E_Remmel}
\journal{O.\ E\v{g}ecio\v{g}lu and J.\ Remmel}
	    {Bijections for Cayley trees, spanning trees,
	     and their $q$-analogues}
	     {J.\ Combin.\ Theory Ser.\ A}
	     {42}{1986}{15--30}

\bibitem{Ehrenborg_Happ}
\journal{R.\ Ehrenborg and A.\ Happ}
	   {Parking cars of different sizes}
           {Amer.\ Math. Monthly}
           {123}{2016}{1045--1048}

\bibitem{Konheim_Weiss}
\journal{A.\ G.\ Konheim and B.\ Weiss}
             {An occupancy discipline and applications}
             {SIAM J.\ Appl.\ Math.}
             {14}{1966}{1266--1274}

\bibitem{Kung_Yan_I}
\journal{J.\ P.\ S. Kung and C.\ Yan}
            {Gon\v{c}arov polynomials and parking functions}
            {J.\ Combin.\ Theory Ser.\ A}
            {102}{2003}{16--37}

\bibitem{Kung_Yan_II}
\journal{J.\ P.\ S. Kung and C.\ Yan}
            {Exact formulas for moments of sums of classical parking functions}
            {Adv.\ in Appl.\ Math.}
            {31}{2003}{215--241}

\bibitem{Kung_Yan_III}
\journal{J.\ P.\ S. Kung and C.\ Yan}
            {Expected sums of general parking functions}
            {Ann.\ Comb.}
            {7}{2003}{481--493}

\bibitem{Riordan}
\journal{J.\ Riordan}
            {Ballots and trees}
            {J.\ Combinatorial Theory}
            {6}{1969}{408--411}

\bibitem{Rota_Kahaner_Odlyzko}
\journal{G.-C.\ Rota, D.\ Kahaner and A.\ Odlyzko}
            {On the foundations of combinatorial theory. VIII.
              Finite operator calculus.}
            {J.\ Math.\ Anal.\ Appl.}
            {42}{1973}{684--760}

\bibitem{Strehl}
\journal{V.\ Strehl}
            {Identities of Rothe--Abel--Schl\"afli--Hurwitz-type}
            {Discrete Math.}
            {99}{1992}{321--340}

\bibitem{Yan}
            {C.\ Yan,}
            {Generalized parking functions, tree inversions,
              and multicolored graphs,}
            {Special issue in honor of Dominique Foata's 65th birthday}
            {\it Adv.\ in Appl.\ Math.}
            {\bf 27} (2001), 641--670.

\end{thebibliography}
\end{document}